\documentclass[11pt]{amsart}
\usepackage{amsmath,amssymb,amsthm}
\usepackage{mathtools}
\usepackage{verbatim}
\usepackage{mathrsfs}
\usepackage{xcolor,enumitem}
\usepackage[colorlinks=true,linkcolor=blue,citecolor={red},pdfpagelabels=false]{hyperref}
\newtheorem{thm}{Theorem}[section]
\newtheorem{lem}[thm]{Lemma}

\newtheorem{prob}[thm]{Open problem}
\newtheorem{con}[thm]{Conjecture}
\newtheorem{prop}[thm]{Proposition}

\newtheorem{cor}[thm]{Corollary}

\newcommand{\thmref}[1]{Theorem~\ref{#1}}
\newcommand{\lemref}[1]{Lemma~\ref{#1}}

\newtheorem{rmk}[thm]{Remark}

\numberwithin{equation}{section}

\makeatletter
\let\@@pmod\pmod
\DeclareRobustCommand{\pmod}{\@ifstar\@pmods\@@pmod}
\def\@pmods#1{\mkern4mu({\operator@font mod}\mkern 6mu#1)}
\makeatother

\theoremstyle{plain}
\newtheorem{theorem}{Theorem}[section]

\theoremstyle{definition}

\newtheorem{example}[theorem]{Example}

\newcommand{\be}{\begin{equation}}
	\newcommand{\ee}{\end{equation}}

\begin{document}
	\baselineskip=17pt
	
	\title[Ramanujan congruences]
	{Ramanujan-style congruences for prime level}
	\author{Arvind Kumar, Moni Kumari, Pieter Moree and Sujeet Kumar Singh}
	
%	\address{Einstein Institute of Mathematics, the Hebrew University of Jerusalem, Edmund
	%	Safra Campus, Jerusalem 91904, Israel.}
%	\email{arvind.kumar@mail.huji.ac.il}
\address{Department of Mathematics, Indian Institute of Technology Jammu, Jagti NH-44, PO Nagrota, Jammu 181221, India.}
\email{arvind.kumar@iitjammu.ac.in}
	\address{Max-Planck-Institut f\"ur Mathematik, Vivatsgasse 7, 
	53111 Bonn, Germany}
%	\address{Department of Mathematics, Bar-Ilan University, Ramat Gan 52900, Israel.}
	\email{kumari@mpim-bonn.mpg.de}
	
	\address{Max-Planck-Institut f\"ur Mathematik, Vivatsgasse 7, 
		53111 Bonn, Germany}
	\email{moree@mpim-bonn.mpg.de }
	
	\address{School of Mathematical Sciences, The University of Nottingham, University Park, Nottingham NG7 2RD, United Kingdom.}
	\email{sujeet170492@gmail.com}

	\date{\today}

	\maketitle
	
	\begin{abstract}
		We establish Ramanujan-style congruences modulo certain primes $\ell$ between an Eisenstein series of weight $k$, prime level $p$ and a cuspidal newform in the $\varepsilon$-eigenspace of the Atkin-Lehner operator inside the space of cusp forms of weight $k$ for $\Gamma_0(p)$. Under a mild assumption, this refines a result of Gaba-Popa. We use these congruences and recent work of Ciolan, Languasco and the third author on Euler-Kronecker constants, to quantify the non-divisibility of the 
		Fourier coefficients involved by $\ell.$ The degree of the number field generated by these coefficients we investigate using recent results on prime factors of shifted prime numbers.
	\end{abstract}
	
	\section{Introduction} 
	Let $E_k$ be the Eisenstein series of even weight $k\ge  2$ for the group $SL_2({\mathbb Z})$,
	normalized so that its Fourier series expansion is
	$$E_k(z)=-\frac{B_k}{2k}+\sum_{n=1}^{\infty}\sigma_{k-1}(n)e^{2\pi i n z},$$
	where $B_k$ is the $k$th Bernoulli number and $\sigma_r(n) = \sum_{d|n}d^{r}$ is the
	$r$-th sum of divisors function.
	The prototype of a Ramanujan congruence goes back 
	to 1916 and asserts that
	\begin{equation}\label{RC}
		\tau(n)\equiv \sigma_{11}(n)  \pmod*{691},
	\end{equation}
	for every positive integer $n$. This can be viewed as a (coefficient-wise) congruence between the unique cusp form $\Delta(z)=\sum_{n=1}^{\infty}\tau(n)e^{2\pi i n z}$ of weight $12$ and the Eisenstein series 
	$E_{12}(z)$, namely
	$	\Delta \equiv E_{12}  \pmod*{691}.$
	There are several well-known ways to prove, interpret, and 
	generalize this. 
	For example, to higher weights eigenforms of level 1
	by Datskovsky-Guerzhoy \cite{dagu}, to newforms  of weight $k$ and prime level $p$  by Billerey-Menares \cite{bime}, and to Fourier coefficients of index coprime to $p$ by Dummigan-Fretwell \cite{dufr}. The latter two authors were primarily motivated
	by an interesting relation of these congruences with the Bloch-Kato conjecture for 
	the partial Riemann zeta function.
	Gaba-Popa \cite{gapo} refined these results, by determining, under some technical conditions, also the Atkin-Lehner eigenvalue of the involved newform, and thus obtained congruences for all coefficients.
	\par 	To make our statements more concrete, we first define 
	for $\varepsilon\in \{\pm 1\}$ an Eisenstein series of 
	even weight $k\ge 2$ and prime level $p,$ 
	namely
	\begin{equation}\label{ES}
		E_{k,p}^{\varepsilon}(z):=E_k(z)+ \varepsilon E_k| W_p(z)={E_k(z)+{\varepsilon}p^{k/2}E_k(pz)}, 
	\end{equation}
	where $W_p$ is the Atkin-Lehner operator.	
	By $M_k^{\varepsilon}(p)$ (resp.\,$S_k^{\varepsilon}(p)$) we denote the $\varepsilon$-eigenspace of the Atkin-Lehner operator $W_p$ inside $M_k(p)$ 
	(resp.\,$S_k(p)$), the space of modular forms (resp.\,cusp forms) of weight $k$ and for the group $\Gamma_0(p)$.
	It is known that $E_{2,p}^{-1}\in M_2^{-1}(p)$ and  
	$ E_{k,p}^{\varepsilon}\in M_k^{\varepsilon}(p)$ for $k\ge 4.$ Using the Fourier series expansion of $E_k$, we obtain
	\begin{equation*}
		E_{k,p}^{\varepsilon}(z)=-\frac{B_k}{2k}\varepsilon(\varepsilon+  p^{k/2})+\sum\limits_{n\ge 1} \left(\sigma_{k-1}(n)+\varepsilon p^{k/2}\sigma_{k-1}\left(\frac{n}{p}\right)\right)e^{2\pi i n z},
	\end{equation*}
	where $\sigma_{k-1}\left(\frac{n}{p}\right)=0$ if $p\nmid n$. We now recall the main result of Gaba-Popa, 
	the proof of which relies on the theory of period polynomials for congruence subgroups developed by Paşol and Popa \cite{papo}.
	\begin{thm}\label{gapo_result} \cite[Theorem 1]{gapo}
		Let $k\ge 4$ be an even integer, $p$ a prime and $\varepsilon \in \{\pm  1\}$. 
		Let $\ell\ge  k+2$ be a  prime such that 
		\begin{center}
			$\ell  \mid \frac{B_k}{2k}(\varepsilon + p^{k/2})$ and
			$\ell  \mid (\varepsilon+p^{k/2}) (\varepsilon+p^{k/2-1}).$ 
		\end{center}
		In case $\ell \nmid (\varepsilon+  p^{k/2})$ we assume 
		in addition that there exists an even integer $n$
		with $0 <n <k$ such that
		$\ell\nmid B_n B_{k-n} (p^{n-1}-1).$
		Then, there exists a newform $f\in S_k^{\varepsilon}(p)$ and a prime ideal $\lambda$ lying above $\ell$ in the coefficient field  of $f$ such that 
		\begin{equation*}
			f\equiv E_{k,p}^{\varepsilon} \pmod*{\lambda}.
		\end{equation*}
	\end{thm}
	%	The proof of the above result relies on the theory of period polynomials for congruence subgroups developed in Paşol and Popa \cite{papo}.

	%The purpose of this paper is twofold. The first objective is to prove the existence of a congruence between an Eisenstein series of weight $k,$ level $p$ and a newform in the $\varepsilon$-eigenspace $S_k^{\varepsilon}(p)$ of the Atkin-Lehner operator $W_p$ inside $S_k(p)$, where $\varepsilon\in \{\pm 1\}$. 
We emphasize that from the latter congruence it follows that if $\overline \rho_{f,\lambda}$ 
			denotes the mod $\lambda$ Galois representation, then (up to semisimplification) $\overline \rho_{f,\lambda} \simeq 1 \oplus \chi_\ell^{k-1}$, where $\chi_\ell$ is the mod $\ell$ cyclotomic character. Therefore, \thmref{gapo_result} gives a sufficient conditon on the prime $\ell$ 
			such that the representation $1 \oplus \chi_\ell^{k-1}$ arises from a newform in $S_k^\varepsilon(p)$.
				
	The purpose of this paper is 
	to strengthen \thmref{gapo_result} and, on
	a somewhat different note, to quantify 
	the non-divisibility of the Fourier coefficients 
	of $E_{k,p}^{\varepsilon}$. The corresponding results
	are presented in the next section, respectively in Section \ref{quanti}.
	\subsection{Strengthening of Theorem \ref{gapo_result}}	
	We sharpen Theorem \ref{gapo_result} in the next two theorems.
	%case $p\not\equiv -1\pmod*{\ell}.$
	
	\begin{thm}\label{main}
		Let $k\ge 2$ be an even integer, $p$ a prime  and $\varepsilon \in \{\pm  1\}$. 
		If $k=2$, we
	also	assume that $\varepsilon=-1.$ 
		Let $\ell\ge \max\{5, k-1\}$ be a prime such that $p\not\equiv -1\pmod*{\ell}$. 
		Then the following are equivalent:
		\begin{itemize}
			\item[$(1)$] $\ell  \mid \frac{B_k}{2k}(\varepsilon + p^{k/2})$ and
			$\ell  \mid (\varepsilon+p^{k/2}) (\varepsilon+p^{k/2-1});$
			\item[$(2)$]
			the existence of a newform $f\in S_k^{\varepsilon}(p)$ and a prime ideal $\lambda$ lying above $\ell$ in the coefficient field  of $f$ such that 
			$$f\equiv E_{k,p}^{\varepsilon} \pmod*{\lambda}.$$
		\end{itemize}
		\end{thm}
	
	\noindent This result improves on Theorem 
	\ref{gapo_result} in three different aspects.
	\begin{itemize}
		\item[(a)] Instead of $\ell > k+1,$ now also  $\ell=k\pm 1$ is allowed. Gaba-Popa \cite{gapo} pointed out that, based on several numerical examples, they expect that their result should hold even for $\ell=k\pm 1,$  although their method breaks down for these values of $\ell$. Therefore, it 
		is reasonable to expect that \cite[Conjecture 3.2]{bime} and \cite[Conjecture on p. 53]{gapo}
		should also hold for $\ell=k\pm 1$.
		\item[(b)] Taking $k=2$ is allowed and hence this recovers an earlier result of Mazur \cite[Proposition 5.12 (ii)]{maz} who used geometry of modular curves associated to  weight 2 modular forms and the $q$-expansion principle to prove his result.
		\item[(c)] There is no condition on $B_n B_{k-n} (p^{n-1}-1)$ anymore in case $\ell \nmid (\varepsilon+  p^{k/2}).$ 
	\end{itemize}
	Comparing \thmref{main} with \thmref{gapo_result}, we see that there
	is now the extra condition $p\not\equiv -1\pmod*{\ell}$ (redundant for $k=2$). In some special cases we remove this condition, together with the assumption $\ell \ge k-1,$ by proving the following variant of  \thmref{main}.  
		\begin{thm}\label{rmain}
			Let $k\ge 2$ be an even integer, $\ell \ge 5$ and $p$ be  primes and $\varepsilon \in \{\pm  1\}$. If $k=2$, we also assume that $\varepsilon=-1$. Suppose that $\ell \mid \frac{B_k}{2k}(\varepsilon+  p^{k/2})$. We further assume that $k\not\equiv0 \pmod{\ell-1}$ and  $\ell\nmid \frac{B_k}{2k}$. 	Then there exists a newform $f\in S_k^\varepsilon(p)$ and a prime ideal $\lambda$ over $\ell$ in the coefficient field of $f$ such that 
			$$
			f\equiv E_{k,p}^{\varepsilon} \pmod*{\lambda}.
			$$		
		\end{thm}
		
	Our proof of the above theorems	uses some classical results from the theory of mod $\ell$ modular forms  and Deligne's theorem on Galois representations attached to eigenforms. 
	Further, it is based on the ideas used in \cite{dufr}, is quite classical in nature, and  completely avoids the use of period polynomials. More precisely, we first prove that the assumptions on $\ell$ ensure that the reduction of $E_{k,p}^{\varepsilon}$ modulo $\ell$ is a cuspidal eigenform in characteristic $\ell,$ and then using the Deligne-Serre lifting lemma we lift it to an eigenform in characteristic zero. In
	the final step we apply the Diamond-Ribet level raising theorem and a result of Langlands to obtain the desired newform.
	It is also worth pointing out that our
	method of proof enables us to obtain a refinement of the celebrated
	Diamond-Ribet level raising theorem (for a precise statement see
	\thmref{refine}).
	
	{We remark that the main underlying ideas in our work are similar to that of  Billerey-Menares \cite{bime}. However, because we are only working over forms of prime level in contrast to squarefree level, we are able to consider more appropriate Eisenstein series $E_{k,p}^\varepsilon$ (cf. \cite[\S 1.2.2]{bime}), resulting an improvement over \cite{bime} and allowing us to take $\ell\ge k-1$. %Finally, it would be of great interest to extend  the above results for squarefree level (see, \cite[Conjecture 3.2]{bime})
	 }

	%Our next result refines \cite[Theorem 1]{dufr} and it can be proved along the same lines as the proof of \thmref{eigenform} (an intermediate result required to prove \thmref{main}) by 	considering the commuting family of operators $\mathcal \{T_q, W_p: q\neq p\},$  instead of $\{T_q, U_p: q\neq p\}$. We leave the details to the interested reader.

	Next using some elementary ideas, we establish the following result in which
	% which are very different from those used to prove \thmref{main}. 
	%in some special cases, it does not require any technical assumptions on $\ell$ and $k$. 
	the resulting cusp form may not be an eigenform as 
	before, but it will always have rational Fourier coefficients. 
	\begin{thm}\label{main2}
		Let $k\ge 2$ be an even integer, $p$ a prime  and $\varepsilon \in \{\pm  1\}$. 
		If $k=2$, we
	also	assume that $\varepsilon=-1.$  Let $N^{\varepsilon}_{k, p}$ be the reduced numerator of ${\frac{B_k}{2k}(\varepsilon+  p^{k/2})}$. Suppose 
		that at least one of the following conditions hold:
		\begin{itemize}
			\item[{\rm (a)}]
			$ p\in \{2, 3, 5, 7, 11, 13, 17, 19, 23, 29,31, 41, 47, 59, 71 \}$.
			\item[{\rm (b)}]
			$k\ge 8,$ $k \equiv 1 - \varepsilon\pmod* 4$ 
			and $N_{k, p}^{-1}$ is coprime to $p-1$.	 	
			\item[{\rm (c)}]
			$k\ge 10$ and $k \equiv 1 - \varepsilon\pmod* {10},$ and $N^{\varepsilon}_{k, p}$ is coprime to $(p+\varepsilon)p(p+1).$ 
		\end{itemize}
	{If the space $S_k^{\varepsilon}(p)$ is non-trivial,} then there exists a non-zero cusp form $f\in S_k^{\varepsilon}(p)$ with rational Fourier coefficients such that
		$$f\equiv E_{k, p}^\varepsilon \pmod*{N_{k, p}^\varepsilon}.$$
	\end{thm}
	\begin{rmk}
		{\rm The primes $p$ in (a) are exactly those for 
			which the genus of the
			Fricke group of level $p$ is zero. They are
			also precisely the prime factors of  
			$2^{46}\cdot 3^{20}\cdot 5^9\cdot 7^6\cdot 11^2\cdot 13^3\cdot 17 \cdot19 \cdot23 \cdot29 \cdot31 \cdot41\cdot 47\cdot59\cdot71,$ the order of the Monster
			group. That is not a coincidence! 
			For more details, see, e.g., Gannon \cite{terry0,terry}}.
	\end{rmk}
{Using \thmref{main2} one can prove 
	\thmref{eigen}. Even though its assertion is weaker than Theorems \ref{main} and \ref{rmain}, we stress that 	the significance of \thmref{eigen} is that it is less restrictive and it gives congruences even for small primes, namely for $\ell=2$ and $3$.}
%	Using \thmref{main2} one can prove 	\thmref{eigen}, which gives congruences even for small primes, namely for $\ell=2$ and $3$.
	The proof 
	is completely patterned after the proof of 	\cite[Lemma 2.1, Theorem 2]{dagu} and so we omit it here. 
	It rests on the fact that the set 
	$$
	\mathcal{B}:=\{f+\varepsilon f|{W_p} : f\in S_k(1)  ~\text{is a normalized eigenform} \} \cup \{g: g \in S_k^\varepsilon(p)~\text{newform}  \}
	$$ 
	forms a basis of $S_k^{\varepsilon}(p)$ consisting of normalized Hecke eigenfunctions for all $T_q$ for $q\neq p$.
	\begin{thm}\label{eigen}
		Suppose at least one of the conditions $(a), (b),  (c)$ of \thmref{main2} holds.       
		Let $\mathcal{B}$ be a basis of $S_k^{\varepsilon}(p)$ of
		normalized Hecke eigenfunctions for all $T_q$ for $q\neq p$. Suppose some prime ideal $\lambda$  divides 
		$N_{k, p}^{\varepsilon}$ in the coefficient field $\mathbb{Q}(a_f(q): f\in \mathcal{B},\,q\neq p)$. Then
		there exists a cusp form $f=\sum_{n\ge 1}a_{f}(n)e^{2\pi i nz} \in \mathcal{B},$ such that for all integers
		$n$ coprime to $p,$ we have
		$$a_{f}(n)\equiv \sigma_{k-1}(n) \pmod*{\lambda}.$$
		%Moreover, if $k<12$ and $\lambda |(\varepsilon+p^{k/2}) (\varepsilon+p^{k/2-1})$ then $f$ is a newform and  $f \equiv E_{k, p}^{\epsilon} \pmod*{\lambda}$.  
	\end{thm}
	As an application of our results (especially of \thmref{eigenform}), we give a non-trivial lower bound of the  degree of the number field generated by all normalized eigenforms (and newforms) in the space $S_k(p)$, see Section \ref{application}. These bounds improve a similar result of \cite{bime} and are  valid in a subset of the primes with natural density close to one. 
	%\begin{cor}\label{eigen}
	%%	$\ell$ be a prime divisor of  $\frac{B_k}{2k}(\varepsilon+  p^{k/2}).$
	%		Then there exists a normalized eigenfunction $f(z)=\sum_{n\ge 1}a_f(n)e^{2\pi i n z} \in S_k^\varepsilon(p)$ for all $T_q$ for $q\neq p$ and {\underline {\rm for any}} prime ideal $\lambda$ over $\ell$ in the coefficient field $\mathbb Q(a_f(q):q\neq p)$ such that for any prime $q\neq p$, we have
	%	$$
	%	a_f(q) \equiv 1+q^{k-1} \pmod*{\lambda}.
	%	$$		
	%\end{cor}
	\subsection{Quantification of Fourier 
		coefficient non-divisibility}
	\label{quanti}
	The second goal of this article is to quantify
	how often $\ell\nmid a(n)$ for certain prime numbers $\ell$ 
	and Fourier coefficients $a(n).$
	This problem
	was first considered by Ramanujan  for his tau function (in \cite{beon} that remained unpublished
	for many years). He made various claims of the form
	\begin{equation}
		\label{valseanalogie}
		\sum_{n\le x,\,\ell\nmid \tau(n)}1=C_{\ell}\int_2^x \frac{dt}{(\log t)^{\delta_{\ell}}}+O\left(
		\frac{x}{(\log x)^r}
		\right),
	\end{equation}
	that he thought to be valid for arbitrary $r$. 
	For example, he claimed
	\eqref{valseanalogie} for $\ell=691$ and $\delta_{\ell}=1/690$.
	Partial integration gives
	\begin{equation}
	\label{partial}
	C_{\ell}\int_2^x\frac{dt}{(\log t)^{1/\delta_{\ell}}}=\frac{C_{\ell} \,x}{(\log x)^{1/\delta_{\ell}}}\left(1+
	\frac{1}{\delta_{\ell}\log x}
	+O_{\ell}\left(\frac{1}{(\log x)^2}\right)\right).
	\end{equation}
	The functions $$\frac{C_{\ell} \,x}{(\log x)^{1/\delta_{\ell}}}\quad\text{and}\quad C_{\ell}\int_2^x\frac{dt}{(\log t)^{1/\delta_{\ell}}},$$
	are now called the \emph{Landau approximation}, respectively \emph{Ramanujan approximation} of
	the counting function in \eqref{valseanalogie}, the true behavior of which is, see Serre \cite{ser76}, 
	\begin{equation}
	\label{true}
	\sum_{n\le x,\,\ell\nmid \tau(n)}1=\frac{C_{\ell} \,x}{(\log x)^{1/\delta_{\ell}}}\left(1+
	\frac{1-\gamma_{\tau;\ell}}{\delta_{\ell}\log x}
	+O_{\ell}\left(\frac{1}{(\log x)^2}\right)\right),
		\end{equation}
	with $\gamma_{\tau;\ell}$ a constant sometimes called \emph{Euler-Kronecker constant}.
	Note that if $\gamma_{\tau;\ell}>1/2$ the Landau approximation asymptotically gives a better approximation to the sum on the left of \eqref{true} than the Ramanujan approximation, and that if 
	$\gamma_{\tau;\ell}<1/2$  it is
	the other way around. Comparing \eqref{partial} and \eqref{true} we see that 
	Ramanujan's claim \eqref{valseanalogie} entails $\gamma_{\tau;\ell}=0.$ For
	$\ell=3,5,7,23$ and $691$ this was disproved by Moree \cite{mor}. For the true value
	of these numerical constants see Table 1 
	(data taken from \cite{clm}).
	
	%\vfil\eject
	
	\centerline{{\bf Table 1:} Euler-Kronecker constants $\gamma_{\tau;\ell}$}
	\label{tab:LvRold}
	\begin{center}
		\begin{tabular}{|c|c|c|}\hline
			$\ell$ & $\gamma_{\tau;\ell}$  & winner \\ \hline \hline
			$3$ & $+0.534921\ldots$ & Landau \\ \hline
			$5$ & $+0.399547\ldots$ & Ramanujan \\ \hline
			$7$ & $+0.231640\ldots$ & Ramanujan \\ \hline
			$23$ & $+0.216691\ldots$ & Ramanujan \\ \hline
			$691$ & $+0.571714\ldots$ & Landau \\ \hline
		\end{tabular}
	\end{center}
	Let $\ell$ be an odd prime.  An arithmetic function $\mathfrak f$ assuming 
	integer values  has 
	a \emph{refined $\ell$-non-divisibility asymptotic} with Euler-Kronecker constant 
	$\gamma_{\mathfrak{f};\ell}$ if there exist
	positive constants $C_{\ell}$ and $h_1$ such that
	\begin{equation}
		\label{refined}
		\sum_{n\le x,\,\ell\nmid {\mathfrak f}(n)}1=\frac{C_{\ell} \,x}{(\log x)^{1/h_1}}\left(1+
		\frac{1-\gamma_{\mathfrak f;\ell}}{h_1 \log x}
		+o_{\mathfrak f}\left(\frac{1}{\log x}\right)\right),
	\end{equation}
	where the implicit constant in the error term may depend on ${\mathfrak f}.$
	\begin{thm}	
		\label{conditionA} Let $\mathfrak{f}$ be an integer valued multiplicative function.
		If
		\begin{equation}
			\label{primecondition}
			\#\{p_1\le x:\,p_1\text{~prime~and~}\ell\mid {\mathfrak f}(p_1)\}=\delta~\sum_{p_1\le x}1+O_{\mathfrak f}\bigg(\frac{x}{
				(\log x)^{2+\rho}}\bigg),
		\end{equation}
		 for some real numbers $\rho>0$ and $0<\delta<1,$ 
		then \eqref{refined} holds with $h_1=1/\delta$ for some 
		positive constant $C_{\ell}.$
	\end{thm}
	\begin{cor}
		Let	$m\ge 1$ be an integer and $\ell$ an
		odd prime such that $h_2:=(\ell-1)/gcd(\ell-1,m)$ is even.
		%it is an easy exercise to show that
		%	\eqref{primecondition} holds with 
		%	$\delta=1/h_1.$ Thus the sum of divisors function $\sigma_m$ has 
		Then the $m$-th sum of divisors function $\sigma_m$ has 	
		a refined $\ell$-non-divisibility asymptotic with
		$h_1=h_2$ for some 
		positive constant $C_{\ell}.$ 
	\end{cor}	
	
	%	Let $\ell$ be an odd prime and $k\ge 1$ an integer. 
	%	Put $r=\text{gcd}(k,\ell-1)$ and $h=(\ell-1)/r$. Assume that $h$ is even. 
	%	Then there are positive constants
	%	$C_{k,\ell}$ and  $\gamma_{\sigma_k;\ell}$ 
	%	only depending on $k$ and $\ell$ such that 
	%	\begin{equation}
	%\label{finalckq}
	%		\sum_{n\le x,\,\ell\nmid \sigma_k(n)}1=\frac{C_{k,\ell} \,x}{(\log x)^{1/h}}\left(1+
	%		\frac{1-\gamma_{\sigma_k;\ell}}{h \log x}
	%		+O_{\ell}\left(\frac{1}{\log ^{2} x}\right)\right).
	%	\end{equation}
	The proof of the corollary is left as an exercise, cf.\,\cite{clm}. 
	We remark that in case $h_2$ is odd, \eqref{refined} takes a more trivial form.
	\begin{thm}\label{non-divisibility}
		Let $k\ge 4$ be an even integer, $p$ a prime and
		$\varepsilon\in \{\pm 1\}$.  
		%Put $f(n)=\sigma_{k-1}(n)+\varepsilon %p^{k/2}\sigma_{k-1}(n/p)$.
		Let $\ell\ge 5$ be a prime
		such that $\varepsilon p^{k/2}\equiv -1\pmod*{\ell}$.
		Set $r=\text{gcd}(\ell-1,k-1).$ Let $g_1$ be the multiplicative
		order of $p^r$ modulo $\ell.$ Put $\mu_p=\ell$ if $g_1=1$ and
		$\mu_p=g_1$ otherwise.
		Then $\mathfrak{f}(n)=\sigma_{k-1}(n)+\varepsilon p^{k/2}\sigma_{k-1}(n/p)$ has 
		a refined $\ell$-non-divisibility asymptotic \eqref{refined} with 
		$h_1=(\ell-1)/r,$
		and Euler-Kronecker constant
		\begin{equation}
			\label{gammarelator}
			\gamma_{\mathfrak{f};\ell}=\gamma_{\sigma_{k-1};\ell}+\left(\frac{\mu_p}{p^{\mu_p}-1}-\frac{(\mu_p-1)}{p^{\mu_p-1}-1}\right)\log p,
		\end{equation}
		provided that $h_1$ is even.
	\end{thm}
	This result reduces the study of $\gamma_{\mathfrak{f};\ell}$
	to that of $\gamma_{\sigma_{k-1};\ell},$ which was
	studied in extenso by Ciolan et al.\,\cite{clm}. They gave a 
	(long and involved) formula for 
	this Euler-Kronecker constant that
	allows one to evaluate it with a certified accuracy of several decimals. The relevant computer programs are
	made available at \url{www.math.unipd.it/~languasc/CLM.html}.
	
	\subsection{Plan for the remainder of the article}
	In Section \ref{pre}, we revisit some basic preliminaries needed to prove \thmref{main}, such as Hecke operators, Artin-Lehner newforms theory, mod $\ell$ modular forms and $\ell$-adic Galois representations associated to modular forms. 
	In Section \ref{proof_main}, we give a proof of \thmref{main} by assuming 
	\thmref{eigenform} (proven in Section \ref{v-proof} along with a
	variant of it).
	In Section \ref{proof_main2}, we prove \thmref{main2} followed by a discussion of
	several interesting numerical examples in Section \ref{example}.  In Section \ref{application},  as an application of our results, we give a non-trivial lower bound  for the degree of the field of coefficients of any normalized eigenform  of fixed weight and level, 
	with Section \ref{dickman} recalling some relevant 
	results on large prime factors of shifted primes.
	Finally,  in Section \ref{pietersproof}, we 
	prove Theorem \ref{non-divisibility}.

	\section{Required Preliminaries}\label{pre}
	\subsection{Notation}
	The letters
	$p,p_1,q$ and $\ell$ will denote prime numbers throughout, except in Section \ref{example}, 
	where $q=e^{2\pi i z}$. For a rational number $\frac{m}{n}$, by $\ell \mid  \frac{m}{n}$, we mean that $\ell$ divides the reduced numerator of $\frac{m}{n}$. 
	Given a newform $f$ we denote
	its $n$-th Fourier coefficient by $a_f(n)$ and its 
	{\em coefficient field} $\mathbb Q(a_f(n): n\ge 1)$ by $K_f$. We say two forms $f$ and $g$ are {\em congruent} mod $\ell$ or mod $\lambda$ if $a_f(n)$ is congruent to 
	$a_g(n)$ for every integer $n$. 
	For notational convenience, we also abbreviate $M_k^{\pm 1}(p)$, $S_k^{\pm 1}(p)$ and $E_{k,p}^{\pm 1}$ by  $M_k^{\pm }(p)$, $S_k^{\pm }(p)$ and $E_{k,p}^{\pm }$, respectively. 
	\subsection{Atkin-Lehner operators and newform theory}
	Let $M_k(N)$ and $S_k(N)$ be the $\mathbb C$-vector space of modular forms and cusp forms, respectively  of even weight $k\ge 2$ and level $N$ with respect to $\Gamma_0(N)$ of trivial nebentypus. These spaces have actions of the Hecke operators $T_1,T_2,\ldots$ which satisfy the following relations: $T_1=1$, $T_{mn}=T_m T_{n}$ if $(m,n)=1$ and for prime powers $q^r$ 
	with $q\nmid N$ we have the recurrence 
	$$T_{q^r}=T_qT_{q^{r-1}}-q^{k-1}T_{q^{r-2}}.$$
	For a prime $p\mid N$, the action of $T_p$ (we will denote it by $U_p$) on $f\in M_k(N)$ is given by 
	$a_{\substack{\\ T_p f}}(n)=a_f(np)$ and 
	such an operator $T_p$ is generally 
	called an {\em $U_p$ operator}. 
	Next, we recall some newform theory from \cite{atle}. A modular form $f\in M_k(N)$ is called a {\em Hecke eigenform} if it is an eigenfunction for all the Hecke operators  $T_q$ for $(q, N)=1$ and $U_p$ for all $p \mid N$. It is a well-known result that if $f$ is a Hecke cusp 
	eigenform, then $a_f(1)\ne 0.$ We say such $f$ is {\em normalized} if $a_f(1)=1.$
	
	Suppose that $p\mid N,$ but $p^2\nmid N.$
	Then there are two ways to embed $S_k(N/p)$ inside $S_k(N)$; one by the identity and the other $f(z)\mapsto f(pz)$ which give rise to a map 
	$$S_k(N/p)\oplus S_k(N/p)\rightarrow S_k(N)\text{~defined~by~}(f,g)\mapsto f(z)+g(pz).$$ The image of this map is called the {\em space of $p$-oldforms} in $S_k(N)$, and is denoted by $S_k(N)^{p-{\rm old}}$. The  orthogonal complement of $S_k(N)^{p{\rm-old}}$ in $S_k(N)$ with respect to the Petersson inner product is called the {\em space of $p$-newforms}, and denoted by $S_k(N)^{p{\rm-new}}$. Finally, in case $N$ is squarefree, we define the {\em space of newforms} 
	$S_k(N)^{{\rm new}}$ as the intersection $\bigcap\limits_{p\mid N}S_k(N)^{p{\rm-new}}$.  A normalized Hecke eigenform $f$ in $S_k(N)^{{\rm new}}$ is called a {\em newform}. 
	\par Let $W_p$
	be the {\em Atkin-Lehner operator} on $M_k(p)$ defined by
	$$f|{W_p}(z)= p^{-k/2}z^{-k}f\Big(\frac{-1}{pz}\Big).$$
	It preserves the space $M_k(p)$ and $S_k(p)$ and also since it is an involution its 
	eigenvalue $\varepsilon$ is in $\{\pm 1\}$. 
	Next we state some standard facts about 
	the operators $T_q~ (q\neq p),$ $U_p,$ and $W_p$ and newforms for the space $S_k(p)$ that can be, for
	example, found in \cite[Lemma 17, Theorem 5]{atle}.
	\begin{lem}\label{alr}	We have the following.	
		\begin{itemize}
			\item[$(1)$] 
			Both $\{T_q, U_p: q\neq p\}$ and	$\{T_q, W_p: q\neq p\}$ are commutating families of operators. 
			\item[$(2)$]
			$ f \in S_k(p)$ is a newform if and only if it is an eigenfunction for all $T_q ~(q \neq p)$, $U_p$ and $W_p$.
			\item[$(3)$]
			If $f\in S_k(p)^{{\rm new}}$ is a newform with Atkin-Lehner eigenvalue $\varepsilon,$ then $a_f(p)=-\varepsilon p^{k/2-1}$.
		\end{itemize}
	\end{lem}
	\subsection{Modular forms with coefficients in a ring A}\label{mr}
	Let $M_k(N,\mathbb Z)\subset \mathbb Z[\![q]\!]$ denote the set of elements of $M_k(N)$ 
	having integer 
	Fourier coefficients at the cusp infinity. For a commutative ring $A$,  we define 
	$$
	M_k(N, A)=M_k(N, \mathbb Z)\otimes_{\mathbb Z}A.
	$$ 
	By the $q$-expansion principle, %\cite{}[12.3.1],
	the map $M_k(N, A)\rightarrow A[\![q]\!]$ is injective, so we may view $M_k(N, A)$ as a submodule of $A[\![q]\!]$.  Note that $S_k(N, \mathbb Z)=M_k(N, \mathbb Z)\cap S_k(N)$. Hence we can define $S_k(N, A)$ similarly, and we identify it with an $A$-submodule of $M_k(N, A)$.
	
	The Hecke operators $T_n$ defined earlier also act on the space $M_k(N, A)$ with 
	the small modification that 
	the action of $T_\ell$ on $M_k(N,A)$ coincides with the action of $U_\ell$ if $A$ is a domain of positive characteristic $r$ and $\ell\mid r.$ 
	\subsection{Mod $\ell$ modular forms}
	For a prime $\ell$, let $\overline{\mathbb{F}}_{\ell}$ denote an algebraic closure of the finite field $\mathbb{F}_{\ell}$, with $\ell$ elements.   In this section we recall the notion of modular forms with coefficients in $\overline{\mathbb F}_{\ell}$ (see  \cite[Section 3.1]{ser87}). 
	
	Fix an embedding $\iota_{\ell}:\overline{\mathbb Q}\hookrightarrow \overline{\mathbb Q}_{\ell}$ in particular we have $\overline{\mathbb Z}\hookrightarrow \overline{\mathbb Z}_{\ell}$.
	%Choose a place of $\overline{\mathbb Q}$ above $\ell$ and let $\overline{\mathbb Z}_{\ell}$ denote the ring of integers of
	%$\overline{\mathbb Q}_{\ell}$. Then the choice of place determines an embedding $\overline{\mathbb Z}\hookrightarrow \overline{\mathbb Z}_{\ell}$.
	 Therefore the ring $\overline{\mathbb Z}_{\ell}$ has a natural reduction map to its residue field
	$\overline{\mathbb F}_{\ell}$ and we obtain a homomorphism 
	\begin{align*}
		\overline{\mathbb Z}_{\ell}&\rightarrow \overline{\mathbb F}_{\ell} {\rm~defined~by~}
		a \mapsto  \overline{a}.
	\end{align*}
	For $k\ge 2$ and an integer $N$, coprime to $\ell$, we define the space of modular forms of type $(N,k)$ with coefficients in $\overline{\mathbb F}_{\ell}$, denoted by $M_k(N, \overline{\mathbb F}_{\ell})$, consisting of formal  power series 
	$$F(z)=\sum_{n\ge 0}A_ne^{2\pi i n z}, ~~A_n \in \overline{\mathbb F}_{\ell},$$
	for which there exists a modular form
	$f(z)=\sum_{n\ge 0}a_ne^{2\pi i n z} \in M_k(N), ~~a_n \in \overline{\mathbb Z},$
	such that $\overline{a}_n=A_n$ for all $n\ge 0$. The space $S_k(N, \overline{\mathbb F}_{\ell})$ is defined analogously. 
	
	As mentioned in Section \ref{mr}, we have the action of the Hecke algebra generated by the operators $T_q$, $q \nmid \ell N$ and $U_{p}$, $p\mid \ell N$ on the space $M_k(N, \overline{\mathbb F}_{\ell})$, they also preserve 
	$S_k(N, \overline{\mathbb F}_{\ell})$.
	Observe that by the Deligne-Serre lifting lemma if 
	$F \in S_k(N, \overline{\mathbb F}_{\ell})$ is a non-zero normalized Hecke eigenform then  $F$ is the reduction$\mod {\ell}$ of some normalized Hecke eigenform	$f\in S_k(N, \overline{\mathbb Z}_{\ell})$.
	
	We say $F \in S_k(N, \overline{\mathbb F}_{\ell})$ is an {\em eigenfunction for the Atkin-Lehner operator} $W_N$ with eigenvalue $\varepsilon$ if it is a reduction of $f \in S_k(N, \overline{\mathbb{Z}}_{\ell})$ which is an eigenfunction for $W_N$ with eigenvalue $\varepsilon$. 
	Notice that this is a well-defined operator on 
	$M_k(N, \overline{\mathbb F}_{\ell})$
	as  we know that if $f$ and $f'$ are characteristic zero modular forms of the same weight and level N  that are congruent modulo $\ell$
	then $W_Nf$ and $W_Nf'$ are congruent modulo $\ell$ as well.

\subsection{Galois representations attached to modular forms}
	In this section, we briefly recall some standard facts  about 2-dimensional Galois representations of Gal$(\overline{\mathbb Q}/\mathbb Q)$ associated to Hecke eigenforms.
	
	Let $f(z)=\sum_{n\ge 1}a_f(n)e^{2\pi i n z}$ be a normalized Hecke eigenform of weight $k$ and level $N$. It is  well-known that the Fourier coefficients $a_f(n)$ belong to the ring of integers $\mathcal{O}_{K_f}$ of a finite extension field $K_f$ of $\mathbb Q$. 
	For a given
	prime $\ell,$ due to a theorem of Deligne, corresponding to such an eigenform $f$ and a prime ideal $\lambda$ above $\ell$ in $K_f$, there is a continuous $\ell$-adic Galois representation 
	$$\rho_{f, \lambda}:{\rm Gal}(\overline{\mathbb Q}/\mathbb Q)\rightarrow GL(2, K_{f,\lambda}),$$
	where $K_{f,\lambda}$ is the completion of $K_f$ at the place $\lambda$. The representation $\rho_{f, \lambda}$ is irreducible, unique up to isomorphism and  it is unramified outside the primes dividing $N$ and the norm of $\lambda$, and has 
	the following properties:
	\begin{center}
		${\rm {tr}}(\rho_{f, \lambda}({\rm {Frob}}_q))=a_f(q)$ and 
		${\rm {det}}(\rho_{f, \lambda}({\rm {Frob}}_q))=q^{k-1},$
	\end{center}
	where Frob$_q\in {\rm Gal}(\overline{\mathbb Q}/\mathbb Q)$ is the Frobenius element at the prime $q$.
	Conjugating by a matrix in $GL(2, K_{f,\lambda})$, one can assume that the image of 
	$\rho_{f, \lambda}$ lands inside $GL(2, \mathcal O_{K_{f,\lambda}})$. Reducing this representation with values in $GL(2, \mathcal O_{K_{f,\lambda}})$ modulo $\lambda$, we get a mod-$\ell$ representation of 
	Gal$(\overline{\mathbb Q}/\mathbb Q)$
	$$\overline{\rho}_{f, \lambda}:{\rm Gal}(\overline{\mathbb Q}/\mathbb Q)\rightarrow GL(2, \mathcal O_{K_f}/\lambda).$$
	The representation $\overline{\rho}_{f, \lambda}$ is well defined up to semi-simplification and depends only on the cusp form $f$ modulo $\lambda.$
	
	\subsection{Diamond-Ribet level raising theorem and a refinement}
	We now recall the following  celebrated result of Ribet \cite{rib} (for weight two and trivial character) and  Diamond \cite{dia} (for higher weight and non-trivial characters), called {\em level raising theorem},
	which gives a criterion for the existence of a congruence between two newforms of the same weight, but 
	of different level. This plays an important role in the proof of our results. 
	\begin{thm}[Diamond-Ribet level raising theorem]\label{level_raising}
		Let $g \in S_k(N)$ be a newform of weight $k \ge 2$ and let $p$ and $\ell$ be distinct primes not dividing $N$ with $\ell \nmid \frac{1}{2}\varphi(N) N p (k-2)!$.  Let $\lambda$ be a prime ideal above $\ell$ in the field generated by the eigenvalues of all eigenforms in $S_k(Np)$ and $S_k(N)$. Then the following are equivalent:
		\begin{itemize}
			\item[$(1)$]
			$a_g(p)^2 \equiv p^{k-2}(1+p)^2  \pmod* \lambda.$ 
			\item[$(2)$]
			There exists a $p$-newform $f\in S_k(Np)$ such that for every prime $q$ coprime
			to $pN,$ 
			$$a_f(q)\equiv a_g(q) \pmod* \lambda.
			$$
		\end{itemize}
	\end{thm}
	In \cite[Theorem 2]{gapo}, Gaba-Popa obtained a refinement of Diamond's level raising theorem and  their proof is (loosely speaking) a part of the proof of their main result. In the same vein, we record the following refinement of the above theorem which also strengthens  \cite[Theorem 2]{gapo} under a mild assumption. 
	\begin{thm}\label{refine}
		Let $k\ge 2$ be an even integer, $p$ a prime, $N$ a positive integer coprime to $p$ and $\varepsilon \in \{\pm  1\}$.
		If $k=2$ we
		assume furthermore that $\varepsilon=-1.$ Suppose  $g(z)=\sum_{n\ge 1}a_g(n)e^{2\pi i n z} \in S_k(N)$ is a newform. 
		Let $\ell\ge  k-1$ be a prime such that $p\not\equiv -1\pmod*{\ell}$ and $\ell\nmid N \phi(N)$ and let $\lambda$ be a prime ideal above $\ell$ in the field generated by the eigenvalues of all eigenforms in $S_k(Np)$ and $S_k(N)$.
		Then the following are equivalent:
		\begin{itemize}
			\item[$(1)$] 
			$a_g(p) \equiv -\varepsilon p^{k/2-1}(1+p) \pmod \lambda$.
			\item[$(2)$]
			There exists an eigenform $f\in S_k^{\varepsilon}(Np)$ which is new at $p$ such that 
			\begin{equation*}\label{maineq}
				f(z) \equiv g(z)+ \varepsilon p^{k/2}g(pz) \pmod*{\lambda}.
			\end{equation*}
		\end{itemize}
	\end{thm}
	\begin{proof}
		We omit the proof because for $N=1$  its proof is the content of the second half of the proof of \thmref{main} (c.f.\,\eqref{ref1}), and for general $N$ the same arguments apply.
	\end{proof}

	\section{Proof of \thmref{main}}\label{proof_main}	
	We give a proof of \thmref{main} using \thmref{eigenform} (the proof 
	of which is given in the next section). We first 
prove that $(1)$ implies $(2)$.
The assumptions  on $\ell $, in view of \thmref{eigenform}, ensure that there is  a normalized eigenform $h(z)= \sum_{n\ge 1}a_h(n)e^{2\pi i n z}\in S_k(p)$ and a prime ideal $\lambda$ above $\ell$ in $K_h$ such that 
\begin{equation}
	\label{f_cong}
	h\equiv E_{k,p}^\varepsilon \pmod* \lambda.
\end{equation}
We now prove that this eigenform $h$ can be replaced by a newform under our assumptions on $\ell$ and $k$.
We distinguish between two cases:\\
%	\begin{itemize}
{\bf{Case (i)}}\label{new}
{\em The eigenform $h$ is a newform.} We will show that the $W_p$-eigenvalue  of $h$ is $\varepsilon$. 
Writing $h|W_p=\delta h$ with $\delta \in \{\pm 1\},$ we obtain $a_h(p)=-\delta p^{k/2-1}$ by \lemref{alr}. Now by considering the $p$-th Fourier coefficients of both 
functions appearing in \eqref{f_cong}, and using the fact that $\ell \mid (\varepsilon+  p^{k/2}) (\varepsilon+ p^{k/2-1})$, we obtain 
\begin{align*}
	-\delta p^{k/2-1}\equiv 1+p^{k-1}+\varepsilon p^{k/2} \equiv -\varepsilon p^{k/2-1}\pmod* \ell.
\end{align*}
Since the prime $\ell $ is odd and different from $p$, 
this proves that $h\in S_k^\varepsilon(p).$
\\
{\bf{Case (ii)}}\label{notnew}
{\em The eigenform $h$ is not a newform.} Then there is a level 1 eigenform $g$ such that 
the corresponding $\ell$-adic Galois representations $\rho_{h,  \Lambda}$ and $\rho_{g, \Lambda}$ are the same, where $\Lambda$ is a prime ideal above $\ell$ in the compositum of coefficient fields of all normalized eigenforms in $S_k(p)$ and $S_k(1)$.  By
\eqref{f_cong}, we have $a_h(q)\equiv 1+q^{k-1} \pmod* \Lambda$ for all $q \neq p.$ A standard application of the Chebotarev density theorem then shows that $\bar\rho_{h, \Lambda}$ is isomorphic to $1 \oplus \chi_\ell^{k-1}$, where $\chi_\ell$ denotes the mod $\ell$ cyclotomic character. 
Thus, we conclude that
\begin{equation}\label{isom}
	\bar\rho_{g, \Lambda} \simeq 1\oplus \chi_\ell^{k-1}.
\end{equation}
Because $g$ is of level 1, the representation $\bar\rho_{g, \Lambda}$ is unramified outside $\ell$. In particular, $\bar\rho_{g, \Lambda}$ and $\chi_{\ell}$ are unramified at the prime $p$. Taking the trace of the image of ${\rm Frob}_p$ on both sides of \eqref{isom} yields
\begin{equation}\label{g_cong}
	a_g(p)\equiv 1+p^{k-1} \pmod* {\Lambda}.
\end{equation}
Now using that $\ell$ divides $1+ p^{k-1}+ \varepsilon p^{k/2}+ \varepsilon p^{k/2-1},$
we infer that
\begin{equation}\label{ref1}
	a_g(p) \equiv -\varepsilon p^{k/2-1}(1+p) \pmod* {\Lambda},
\end{equation}
which gives
$$
a_g(p)^2 \equiv p^{k-2}(1+p)^2 \pmod* {\Lambda}.
$$
Hence, by  apply  \thmref{level_raising} 
and using the hypothesis $\ell > k-2,$ we obtain a newform $f\in S_k(p)$  for which 
\begin{equation*}
	a_f(q) \equiv a_g(q) \pmod* {\Lambda}, {\rm~~ for~all~} q\neq p.
\end{equation*}
Taken together with \eqref{g_cong} this results in
\begin{equation}\label{h_cong}
	a_f(q) \equiv 1+q^{k-1} \pmod* {\Lambda}, {\rm~~ for~all~} q\neq p.
\end{equation}
Let $\delta$ be the $W_p$-eigenvalue of $f$ and so  $a_f(p)=-\delta p^{k/2-1}$, from \lemref{alr}. In order to complete the proof of \thmref{main}, we only need to show that $\delta=\varepsilon$. The reason is that if $\delta=\varepsilon,$ then 
$
a_f(p)= -\varepsilon p^{k/2-1}\equiv 1 + p^{k-1}+\varepsilon p^{k/2} \pmod* {\Lambda}.
$
Combining this with \eqref{h_cong} and using that $E_{k,p}^\varepsilon$ mod $\ell$ is an eigenform, which has been established in the course of the proof of \thmref{eigenform}, gives that $f\equiv E_{k,p}^\varepsilon \pmod \Lambda,$
thus completing the proof (note that we can restrict $\Lambda$ to $K_f$ to get the required prime ideal above $\ell$ in $K_f$). 
\par Let $G_p$ denote the decomposition group of the absolute Galois group Gal$(\overline{\mathbb Q}/\mathbb Q)$ at a place over $p.$ Denote, for 
any algebraic integer $\alpha,$ the unique 
unramified character $G_p \rightarrow \overline{\mathbb F}_\ell^\times$
sending the arithmetic 
Frobenius ${\rm Frob}_p$ to $\alpha$ mod $\ell$ by $\mu_{\alpha}.$
It follows from the work of Langlands \cite{lan} (see also \cite[Proposition 2.8 (2)]{lowe})
that the restriction of $\bar\rho_{f, \Lambda}$ to $G_p$ is given by
$$
\bar\rho_{f, \Lambda}|_{G_p} \simeq \begin{pmatrix}
	\chi_\ell^{k/2}  & * \\   & 	\chi_\ell^{k/2-1}
\end{pmatrix} \otimes \mu_{a_f(p)/p^{k/2-1}}.
$$
Since $\mu_\alpha$ and $\chi_\ell$ are unramified at $p,$ one can consider the trace of Frob$_p$ on the right hand side and hence ${\rm tr} (\bar\rho_{f, \Lambda}({\rm Frob}_{p}))$ is well-defined.
Since $\bar\rho_{f, \Lambda} \simeq \bar\rho_{g, \Lambda}$ (up to semisimplification),  we have
$
\bar\rho_{f, \Lambda}|_{G_{p}} \simeq \bar\rho_{g, \Lambda}|_{G_{p}}.
$
Taking the trace of the image of ${\rm Frob}_p$ yields 
$$
\left(p^{k/2}+ p^{k/2-1} \right)\frac{a_f(p)}{p^{k/2-1}} \equiv a_g(p) \pmod* {\Lambda}.
$$
The congruence  \eqref{g_cong} together with $a_f(p)=-\delta p^{k/2-1}$ gives
$$
-\delta(p^{k/2}+ p^{k/2-1}) \equiv 1 +p^{k-1} \equiv -\varepsilon (p^{k/2}+ p^{k/2-1}) \pmod* {\Lambda},
$$
which yields
$
\ell \mid (\delta-\varepsilon)p^{k/2-1}(p+1).
$
Since by assumption $(\ell, p(p+1))=1,$ 
we infer from this that $\delta =\varepsilon,$ and so (1) implies (2).

It remains to show that $(2)$ implies $(1)$. Let $\ell \ge 5$ be a prime for which there exists a newform 
$f(z)=\sum_{n\ge 1}a_f(n)e^{2\pi i n z}\in S_k^\varepsilon(p)$ such that 
\begin{equation}\label{fe}
	f\equiv E_{k,p}^\varepsilon \pmod* \lambda,
\end{equation}
for some prime ideal $\lambda$ above $\ell$ in ${K_f}$. Since the constant term 
of $f$ is zero and the norm of $\lambda$ is a power of $\ell,$ 
we get $\ell \mid  \frac{B_k}{2k}(\varepsilon+  p^{k/2})$. The fact that $f$ is a newform with $W_p$-eigenvalue  $\varepsilon,$ along with \lemref{alr} yields $a_f(p)=-\varepsilon p^{k/2-1}$. Taken together with the congruence \eqref{fe} at the prime $p,$ this leads to 
$-\varepsilon p^{k/2-1}\equiv 1+ p^{k-1}+ \varepsilon p^{k/2} \pmod* \lambda.$
Since $\lambda$ is a prime ideal above $\ell,$ this completes the proof. \hfill {$\square$}

	\section{Variants of \thmref{main} and proof of \thmref{rmain}}\label{v-proof}
	As promised in the previous section, we now give a proof of \thmref{eigenform}.
%We first prove the following result where the goal  is to find an eigenform $h\in S_k(p)$ which is congruent to $E_{k,p}^\varepsilon$.  
This result may be of independent interest because of the limited assumptions on $\ell$ when compared with \thmref{main}. 
%This result play a main role in the proof of \thmref{main}. 
Recall that \thmref{rmain} is an immediate consequence of \thmref{eigenform}. 
We also remind the reader that by ``eigenform" we mean an eigenfunction of all the Hecke operators $T_n$, $n\ge 1$.
\begin{thm}\label{eigenform}
	Let $k\ge 2$ be an even integer, $\ell \ge 5$ and $p$ be  primes and $\varepsilon \in \{\pm  1\}$. If $k=2$, we also assume that $\varepsilon=-1$. Suppose that $\ell$ divides
	both $\frac{B_k}{2k}(\varepsilon+  p^{k/2})$ and
	$(\varepsilon+  p^{k/2}) (\varepsilon+ p^{k/2-1}).
	$
	Then there exists a normalized eigenform $h\in S_k(p)$ and a prime ideal $\lambda$ over $\ell$ in the coefficient field of $h$ such that 
	$$
	h\equiv E_{k,p}^{\varepsilon} \pmod*{\lambda}.
	$$		
	Moreover, if $k=2$, then $h\in S_2^-(p)$ is a newform.
\end{thm} 
\begin{proof}
	We observe that $E_{k,p}^{\varepsilon}|W_p =\varepsilon E_{k,p}^{\varepsilon}$ and 
	that $W_p$ interchanges both the cusps of $\Gamma_0(p),$ which implies that the constant term of $E_{k,p}^{\varepsilon}$ at both the cusps is $-\frac{B_k}{2k}\varepsilon (\varepsilon+  p^{k/2})$, up to a sign and powers of $p$. Since  $\ell \mid  \frac{B_k}{2k}(\varepsilon+  p^{k/2})$, it follows from 
	the $q$-expansion principle (here $q=e^{2\pi iz}$) that the reduction of $E_{k,p}^{\varepsilon}$ modulo $\ell$ gives rise to an element $\overline{E}_{k,p}^{\varepsilon} \in S_k^{\varepsilon}(p, \mathbb F_\ell) \subset S_k(p,\overline{\mathbb F}_{\ell})$. As $E_k$ is an eigenfunction for all the Hecke operators $T_q ~(q \neq p),$ all of which commute with $W_p$, we see that
	$E_{k,p}^{\varepsilon}$, hence $\overline{E}_{k,p}^{\varepsilon}$ is a common eigenfunction of all $T_q$ ($q\neq p$). We next claim that the assumptions on the prime  $\ell$ ensure that $\overline{E}_{k,p}^{\varepsilon}$ is also an eigenfunction of the operators $U_p$. For $k=2$, it is easy to see that $U_p E_{2,p}^-= E_{2,p}^-$
	which shows that ${E}_{2,p}^-,$ and so in particular  $\overline {E}_{2,p}^-$, is an eigenfunction for $U_{p}$ with eigenvalue 1. For $k\ge 4$, a simple computation gives that if $a(n)$ and $b(n)$ are the  $n$th Fourier coefficients of $E_{k,p}^\varepsilon(z),$ 
	respectively $U_{p} E_{k,p}^\varepsilon(z)$, then 
	$$
	b(n)=\begin{cases}
		a(0) & \text{if~}\,n=0;\\
		(1+p^{k-1}+\varepsilon p^{k/2}) a(n) & \text{if~}p\nmid n;\\
		(1+p^{k-1}+\varepsilon p^{k/2}) a(n) -\varepsilon \sigma_{k-1}(n/p)p^{k/2}(\varepsilon+  p^{k/2}) (\varepsilon+ p^{k/2-1}) & \text{otherwise}.
	\end{cases}
	$$
	It shows that $E_{k,p}^\varepsilon$ is not an eigenfunction of $U_p,$ but 
 by assumptions  $\ell  \mid \frac{B_k}{2k}(\varepsilon+  p^{k/2})$ and $\ell \mid (\varepsilon+  p^{k/2}) (\varepsilon+ p^{k/2-1}),$  it follows that
	\begin{equation}\label{eisenstein_eigenform}
	U_p E_{k,p}^\varepsilon(z)\equiv (1+p^{k-1}+\varepsilon p^{k/2}) E_{k,p}^\varepsilon(z) \pmod* \ell.
		\end{equation}
	In other words,  $\overline{E}_{k,p}^\varepsilon$ is an eigenfunction of $U_p$ with eigenvalue $1+p^{k-1}+\varepsilon p^{k/2}$ and this proves our claim.
	
	The reduction map from $S_k(p,\bar{\mathbb Z}_\ell )$ to $S_k(p,\overline{\mathbb F}_\ell)$  is surjective by Carayol's lemma \cite[Proposition 1.10]{edi}. Hence, there exists an element in $S_k(p,\mathcal O_K)$
	having  $\overline{E}_{k,p}^\varepsilon$ as mod $\ell$ reduction, where $\mathcal O_K$ is the ring of integers of some finite extension $K$ of $\mathbb Q_\ell$. 
	In other words, $\overline{E}_{k,p}^\varepsilon$ is the reduction of a characteristic $0$ cusp form, %${g}^\varepsilon$, 
	which  may not be an eigenfunction for the Hecke operators. 
	%To be summarised,
	%	If $\mathbb F$ is the residue field of $\mathcal O_K$, then  $\overline{E}_{k,p}^\varepsilon \in S_k^\varepsilon(p, \mathbb F)$ is a reduction of a cusp form and it is a normalized common eigenfunction for every element of the commuting family $\mathcal T =\{T_q, U_p: q\neq p\}$. 
	Now we use the Deligne-Serre lifting lemma \cite[Lemma 6.11]{dese} guaranteeing  
	the existence of an $h'\in S_k(p,\mathcal O_L)$ that is a normalized common eigenfunction for every element of $\{T_q, U_p: q\neq p\},$
	such that 
	\begin{center}
		$h'\equiv E_{k,p}^\varepsilon\pmod* {\lambda'}$,
	\end{center}
	for some prime ideal $\lambda'$ lying above $\ell$ in $\mathcal O_{L}$. Here $L\supseteq K$ is a finite extension of $\mathbb Q_\ell$ which is a completion of some number field at a prime over $\ell$. Moreover, such an $h'$ arises from some $h(z)= \sum_{n\ge 1}a_h(n)e^{2\pi i n z}\in S_k(p)$ via the embedding of $K_h$ into $L,$ and hence there exists a prime ideal $\lambda$ above $\ell$ in $K_h$ such that 
	\begin{equation*}
		h\equiv E_{k,p}^\varepsilon \pmod* \lambda.
	\end{equation*}
	This proves the first part of \thmref{eigenform}.
	Next, if $k=2$, then we know that $S_2(1)=0,$ and therefore there are no oldforms in $S_2(p),$ and hence $h$ must be a newform.  Let $\delta$ be the $W_p$-eigenvalue 
	of $h.$ Then by \lemref{alr}
	we have $a_h(p)=-\delta,$ whereas the $p$-th Fourier coefficient of $E_{2,p}^{-1}$ is 1. This gives rise to $-\delta \equiv 1\pmod* \lambda,$ and hence $h\in S_2^-(p),$ as $\ell$ (the characteristic of $\lambda$) is odd.  
\end{proof}

The next result refines \cite[Theorem 1]{dufr} and it is a direct consequence of \thmref{eigenform} and \cite[Theorem 1]{dagu}. 
\begin{cor}\label{eigenform_varepsilon}
	Let $k\ge 2$ be an even integer, $p$ a prime  and $\varepsilon \in \{\pm  1\}$. 
	If $k=2$, we
	also	assume that $\varepsilon=-1.$  
	Let 
	$\ell\ge 5$ be a prime divisor of  $\frac{B_k}{2k}(\varepsilon+  p^{k/2}).$
	Then there exists a normalized eigenfunction $f \in S_k^\varepsilon(p)$ for all $T_q$ with $q\neq p,$ and a prime ideal $\lambda$ over $\ell$ in the coefficient field of $f$ such that %for 	any integer $n$ coprime to $p,$ we have
	$$f\equiv E_{k,p}^{\varepsilon} \pmod*{\lambda}.$$	
\end{cor}
\begin{proof}
	If $\ell \mid \frac{B_k}{2k}$, applying \cite[Theorem 1]{dagu} gives a normalized eigenform $g\in S_k(1)$ and a prime ideal $\lambda$ above $\ell$ such that $g\equiv E_k \pmod* \lambda$. In this case, the form $f:=g+\varepsilon g|W_p$ is a desired eigenfunction. We now assume that $\ell\nmid \frac{B_k}{2k},$ which implies  $\ell \mid (\varepsilon+  p^{k/2})$. Then, by \thmref{eigenform}, we obtain a normalized eigenform $h\in S_k(p)$ and a prime ideal $\lambda$ above $\ell$ in the coefficient field of $h$ such that $h\equiv E_{k,p}^\varepsilon \pmod* \lambda$.  If $h$ is a newform, then Case (i) in the proof of \thmref{main} gives $h\in S_k^\varepsilon(p)$ and hence we are done. Whereas, if $h$ is not a newform, then following the arguments in Case (ii) and by \eqref{g_cong}, we have a normalized eigenform $g\in S_k(1)$ of level $1$ such that $g\equiv E_k \pmod*{\lambda},$ where $\lambda$ is a prime ideal above $\ell$ in the coefficient field of $g$. As before, $f:=g+\varepsilon g|W_p$ serves our purpose.
\end{proof}
	\subsection*{Proof of \thmref{rmain}}
		The proof uses \thmref{eigenform} and some ideas used in the proof of \thmref{main}. So we only give a sketch here. We first notice that $\ell  \mid (\varepsilon + p^{k/2})$ and hence  from \thmref{eigenform} we have an eigenform $h\in S_k(p)$ such that $h\equiv E_{k,p}^{\varepsilon}\pmod*{\lambda}.$ We now claim that because of the conditions $\ell\nmid \frac{B_k}{2k}$ and $k\not\equiv0 \pmod{\ell-1}$, the eigenform $h$  has to be a newform with $W_p$-eigenvalue $\varepsilon$, i.e., the Case (ii) in the proof of \thmref{main} does not occur at all. Suppose it does occur, then by \eqref{g_cong} we have an eigenform $g$ of level $1$ such that 
		$g\equiv E_k \pmod*{\Lambda}.$ Since $g$ has integral Fourier coefficients in its coefficient field and $k\not\equiv0 \pmod*{\ell-1}$, applying \cite[Chapter X, Theorem 8.4]{lang} gives $\ell \mid  \frac{B_k}{2k}$, which is a contradiction.
	
	\section{Proof of \thmref{main2}}\label{proof_main2}
	The main idea of the proof is to construct, in each
	of the three cases, a  modular form $g \in M^{\varepsilon}_k(p)$ 
	with integral Fourier coefficients and having  non-zero 
	constant term $a_g(0)$,
	coprime to $N^{\varepsilon}_{k,p},$ such that if one would define
	\begin{equation}\label{f}
		f:=E_{k, p}^{\varepsilon} +\frac{B_k}{2k}\varepsilon (\varepsilon+p^{k/2})\frac{g}{a_g(0)},
	\end{equation}
	then such $f$ should be a non-zero cusp form. The existence of such $g$ then ensures that 
	the corresponding $f$ has rational Fourier coefficients and, moreover,
	$$f\equiv E_{k, p}^\varepsilon \pmod*{N_{k, p}^\varepsilon}.$$
	\subsection{Case (a)}
	Let
	dim$M_k^\varepsilon(p)=d+1$. From  \cite{chki13} for $\varepsilon=+1$ and   \cite{ckl} for $\varepsilon=-1$, we know that for such choices of primes $p$ there exists a 
	basis $\{f_0, f_1,\cdots , f_d\}$ of $M_k^{\varepsilon}(p),$ known as Victor-Miller basis,
	such that all the Fourier coefficients of the $f_j$ are integers 
	and have  Fourier series expansion of the form
	$$
	f_j(z)=e^{2\pi ijz}+O(e^{2\pi i(d+1)z})~~~{\rm for~} 0\le j\le d.
	$$ 
	So, the obvious choice for $g$ in this case is $f_0,$ completing the proof.

	\subsection{Case (b)}
	For any even integer $\alpha\ge 4$ observe that $G_{\alpha}(z)G_{\alpha}(pz)\in M_{2\alpha}^+(p)$, where
	$G_{\alpha}$ is the  Eisenstein series of weight $\alpha$ defined by
	$$
	G_\alpha(z)= 1-\frac{2\alpha}{B_\alpha}\sum_{n\ge 1}\sigma_{\alpha-1}(n)e^{2\pi i n z}.
	$$
	We first consider the situation when $\varepsilon=+1$ and $k \ge 8$ with $k\equiv 0 \pmod* 4$. Such $k$ can be written as $k=8a+12b=4\beta$ for some non-negative integers $a,$ $b$ and $\beta \ge 2,$ 
	and so  
	$$g(z):=(G_4(z)G_4(pz))^a(G_6(z)G_6(pz))^b \in M_{4\beta}^+(p)$$ 
	has integer Fourier coefficients (as $G_4$ and $G_6$ have integer Fourier coefficients) with constant term $1$.
	Next we claim that the corresponding function $f$ defined by \eqref{f} is non-zero by showing that its $e^{2\pi iz}$ coefficient is non-zero, i.e., 
	$$1+(30a-63b)\frac{B_{4\beta}(1+p^{2\beta})}{\beta} \neq 0.
	$$ 
	For that, we write 
	$\displaystyle{\frac{B_{4\beta}}{8\beta}=\frac{m}{n}}$, where $m\in \mathbb{N}, n\in \mathbb{Z}$ and $(m,n)=1$ and hence  $\displaystyle{1+(30a-63b)\frac{B_{4\beta}(1+p^{2\beta})}{\beta}}$ is zero 
	if and only if $n=8(-30a+60b)(1+p^{2\beta})$ and $m=1$. 
	But one can easily check that $n\neq -8(30a-60b)(1+p^{2\beta})$ for $\beta=2$ and $m>1$ for $\beta\ge 3.$
	\par For $\varepsilon=-1$ and $k \ge 8$ with $k\equiv 2 \pmod* 4$, we write $k=2+8a+12b=2+4\beta$. Define
	$$g(z)=G_{2,p}^-(z)(G_4(z)G_4(pz))^a(G_6(z)G_6(pz))^b \in M_{2+4\beta}^-(p),$$ 
	where $ G_{2,p}^-(z)=G_2(z)-pG_2(pz)\in M_2^-(p)$. This $g$ has integer Fourier coefficients and constant term $1-p$.
	The first Fourier coefficient of the corresponding $f$ is  $$1+(240a-504b-24)\frac{B_{2+4\beta}(p^{1+2\beta}-1)}{2(2+4\beta)},$$ and as before one can verify that it is non-zero.
	
	\subsection{Case (c)}
	For $\varepsilon=+1$, write $k=10 \alpha$ with $\alpha >0$.
	Define 
	$$g(z):=(p^2G_4(pz)G_6(z)+p^3G_4(z)G_6(pz))^{\alpha}.$$
	Observe that $p^2G_4(pz)G_6(z)+p^3G_4(z)G_6(pz) \in M_{10}^+(p)$ and 
	hence $g\in M_{k}^+(p).$ Further, $g$ has integer Fourier coefficients and constant term $(p^2+p^3)^{\alpha}.$ Since 
	by assumption, $p(p+1)$ does not divide $N^+_{k,p},$ it follows
	that the corresponding $f$ defined by \eqref{f} satisfies
	$$f\equiv E_{k, p}^\varepsilon \pmod*{N_{k, p}^\varepsilon}.$$
	To show that $f$ is non zero we prove  that its $e^{2\pi iz}$ coefficient is
	non-zero, i.e.,
	$$
	1+\frac{(240p^3-504p^2)\alpha (p^2+p^3) B_{10\alpha}(1+p^{5\alpha})}{20\alpha (p^2+p^3)^{\alpha}} \neq 0.
	$$ 
	First suppose that $\alpha>1$. Write $$\frac{B_{10\alpha}(1+p^{5\alpha})}{20\alpha(p^2+p^3)^{\alpha-1}}=\frac{m}{n}\text{~with~}(m, n)=1.$$  Then one checks that  $m$ is always greater than 1, 
	and therefore the coefficient can not be equal to zero.
	If $\alpha=1$ we write $B_{10}(1+p^{5})/20=m/n$ with $(m, n)=1$. Again the coefficient is zero if and only if $n=-(240p^3-504p^2) m$. Since $(m, n)=1$, we must have $m=1$. But for $\alpha =1$, we have $B_{10}(1+p^{5})/20=(1+p^5)/264$, i.e.,   if $p\ge 5$ then it can be easily seen that $m>1.$ For the remaining two primes we have
	$n\neq -(240p^3-504p^2)$.
	
	Finally we assume that $\varepsilon=-1$ and $k=2+10 \alpha$, for some $\alpha$, integer. Define $$g(z):=G_{2,p}^-(z)(p^2G_4(pz)G_6(z)+p^3G_4(z)G_6(pz))^{\alpha}.$$
	This $g$ has all the required properties and the corresponding $f$, defined by \eqref{f}, has $e^{2\pi iz}$ coefficient 
	$$1+\frac{(\alpha(240p^3-504p^2)(p^2+p^3)-24)B_{2+10\alpha}(1+p^{1+5\alpha})}{2(1+10\alpha)(p^2+p^3)^{\alpha}}.$$
	As before, one can easily check that this coefficient is non-zero and hence we are done. \hfill {$\square$}

	\section{Numerical Examples}\label{example}
	In this section, we give several numerical examples of Ramanujan-style congruences
	and  write $q$ for $e^{2\pi iz}.$
	We first recall some basic facts that will be used. To prove that two normalized eigenforms of weight $k$ and level $p$ are congruent, it is enough to check that their first $k(p +1)/12$  Fourier coefficients at prime indices are congruent (this is due to 
	the Sturm bound). Moreover, if $f(z)=\sum_{n\ge 1}a_f(n)q^n$ is a newform and $\sigma \in {\rm Gal}(\overline {\mathbb Q}/ \mathbb Q)$, then its Galois conjugate $f^\sigma(z):=\sum_{n\ge 1}\sigma(a_f(n))q^n$ is a newform. In fact, it is easy to see that if two modular forms $f$ and $g$ are congruent modulo some prime ideal $\lambda$ above $\ell$ and $\sigma \in {\rm Gal}(\overline {\mathbb Q}/ \mathbb Q)$, then $\sigma(\lambda)$  is a prime ideal above $\ell$ and
	\begin{equation}\label{cong_galois}
	f^\sigma \equiv g^\sigma \pmod {\sigma(\lambda)}.
		\end{equation}
	 %Also, if two forms are congruent modulo a prime ideal above $\ell,$ then then their Galois conjugates are also congruent modulo some prime ideal above $\ell$.
	  To simplify the notation in this section, we put
	\begin{center}
		$N_{k,p}^\varepsilon:=$ the reduced numerator of ${\frac{B_k}{2k}(\varepsilon+  p^{k/2})}~~~$ and
		$~~~M_{k,p}^\varepsilon:=(\varepsilon+p^{k/2}) (\varepsilon+p^{k/2-1})$.
	\end{center}
	
	\begin{example}
{\rm 	Take $p=11$ and $k=6$.  For $\varepsilon=-1$, we easily see that
	$N_{6,11}^-=5\cdot 19$ and both the primes $5$ and $19$ divide
	$M_{6,11}^-$, so  $\ell =5$ or $19$.
	We see that $S_{6}^-(11)$ is 3-dimensional, spanned by the newforms
	\begin{align*}
		g(z) &=  q + aq^2 - (1/6a^2 + 5/3a - 64/3)q^3 + (a^2 - 32)q^4 - (3/2a^2
		+ 7a - 98)q^5 \\
		& \hspace{30pt} - (5/3a^2 - 19/3a - 94/3)q^6+ O(q^{7}),
	\end{align*}
	 where $a$ is any root of  the polynomial $x^3 - 90x + 188.$
	{Because any two newforms in $S_{6}^-(11)$ are Galois conjugates,  in view of \eqref{cong_galois}, \thmref{main} ensures the existence of  a congruence between the newform $g$ and
	$E_{6,11}^-$ modulo some prime in $\mathbb Q(a)$ above $5$ (resp. for $19$)
	%  Moreover, since any two newforms are Galois conjugate,  there  exists a congruence between $g_i$ and $E_{6,11}^-$ modulo some prime in $\mathbb Q(a_i)$ above $5$ and $19$ for each $i=1, 2, 3$,
	 which we verify now.} Factoring the ideals (5) and (19) in the ring of
	integers of $\mathbb Q(a)$ gives $5=\lambda \lambda'$ and
	$19=\beta\beta'^2$, where  $\lambda=(5, -1/6a^2 + 1/3a + 28/3)$,
	$\lambda'= (5, 1/6a^2 + 2/3a - 28/3)$, $\beta=(19, 1/6a^2 + 2/3a - 31/3)
	$ and $\beta'=(19, 1/6a^2 + 2/3a - 58/3)$.
	We then check that 
	\begin{equation}\label{l=5}
		g\equiv E_{6,11}^- \pmod* {\lambda'} \hspace{10pt}{\rm~and}     
		\hspace{10pt}   g\equiv E_{6,11}^- \pmod* {\beta}
	\end{equation}
	We emphasize the fact that the  congruence \eqref{l=5} for the
	prime above $\ell =5=6-1$ is new 
	and corresponds to the non-covered case $\ell=k-1$ in  \cite{gapo}.
	
	For $\varepsilon=+1$, we have $N_{6,11}^+=37$ and $37 \mid M_{6,11}^+.$ As $S_{6}^{+}(11)$
	is 1-dimensional and spanned by the newform
	$$f(z)=q - 4q^{2} - 15q^{3} - 16q^{4} - 19q^{5} + 60q^{6}  + O(q^{7}),
	$$
	\thmref{main} guarantees the congruence               
	$$
	f\equiv E_{6,11}^+ \pmod* {37}.
	$$
	As $11^3\equiv -1\pmod*{37},$ the conditions of 
	\thmref{non-divisibility} are satisfied with 
	$\ell=37,p=11,k=6$ and $\varepsilon=1.$ On using that 
	$\gamma_{1,37}=0.47464\ldots,$
	we conclude that $f$ has 
	a refined $\ell$-non-divisibility asymptotic \eqref{refined} with 
	$h_1=36,$ and Euler-Kronecker constant
	$$\gamma_{f;37}=\gamma_{1,37}+\frac{6}{11^6-1}-\frac{5}{11^5-1}
	=0.47464\ldots-0.000027\ldots=0.47461\ldots$$
	Hence, in this case the Ramanujan approximation is better than the
	Landau one.}
	\end{example} 
	
		\begin{example}
			{\rm  
	Although \thmref{main} holds for $\ell>k-2$, we have checked several numerical examples for the case $\ell \le k-2$ and in all those cases, we find that the assertion of \thmref{main} is true. We give one such example here.	
	Take $k=12$ and $p=7$. Then the space of newforms in $S_{12}^{+}(7)$ is 3-dimensional and spanned by 
	the newforms
	\begin{align*}
		f(z)& =q + aq^2 - (11/21 a^2 - 103/7 a - 33758/21)q^3 + (a^2 - 2048)q^4 + (59/7a^2 - 517/7a - 203864/7)q^5 \\
		&\hspace{30pt}+ (-538/21a^2 + 788/7a + 2476144/21)q^6 +O(q^7),
	\end{align*}
	where $a$ is any root of the polynomial $x^3 - 77x^2 - 2854x + 225104.$
	Here $N_{12,7}^+=5 \cdot 181 \cdot 691$ and $5\cdot 181  \mid  M_{12,7}^+$. Therefore, by using the same reasoning used in the previous example, \thmref{main} guarantees the existence of the congruence of $f$ with $E_{12,7}^+$ modulo a prime 
	ideal above 181 in $\mathbb Q(a),$ something we have verified by
	direct computation.
Note that \thmref{main} is not applicable for the prime $\ell =5,$ as $\ell < k-1=11$. But factoring the ideal (5) in the ring of integers of $\mathbb Q(a)$ gives $5=\lambda \lambda'$, where  $\lambda=(5, -1/14a^2 + 23/14a + 1628/7)$ and $\lambda'=(5, 1/42a^2 - 3/14a - 1775/21)$ and we check that
	$$
	f\equiv E_{12,7}^+ \pmod* {\lambda'}.
	$$
}
	\end{example} 

	\begin{example}
		{\rm  
	The dimension of $S_k^+(p)$ turns out to be $1$  for $p=2, 3, 5, 7, 11$ and for certain finite values of $k,$ and the unique newforms have integer Fourier coefficients. In these cases, \thmref{main2} gives a suitable congruence modulo $N_{k,p}^+$. For example if $p=2$ and $k=8,$ then one easily check that $N_{8,2}^{+}=17$. Then
	$S_8^+(2)$  is 1-dimensional and spanned by the newform 
	$$\Delta_{8,2}(z):=\eta^8(z)\eta^8(2z) \in S_8^+(2).$$
	\thmref{main2} gives that a constant multiple of  $\Delta_{8,2}$ and $E_{8,2}^+ $ are the same modulo 17. By comparing the Fourier coefficients, we see that the constant must be 1 and hence 
	$\Delta_{8,2} \equiv E_{8,2}^+ \pmod* {17}.$}
	\end{example} 
	
	\begin{example}
		{\rm 	This example, which falls outside the scope of \thmref{main}, gives a congruence modulo 6 using \thmref{main2}. Take $k=4$ and $p=17,$ and hence $N_{4,17}^-=6$.  The space $S_4^{-}(17)$ is 1-dimensional, spanned by the newform
	$$
	f(z)= q - 3q^{2} - 8q^{3} + q^{4} + 6q^{5} + 24q^{6} - 28q^{7} + 21q^{8} + 37q^{9} + O(q^{10}).
	$$  
	\thmref{main2} guarantees a congruence  between a constant multiple of $f$ and $E_{4,17}^-$ modulo 6. An easy computation shows that the first 6 
	Fourier coefficients of $f$ and $E_{4,17}^-$ are 
	the same modulo 6. Hence, invoking the Sturm bound, 
	$$
	f \equiv E_{4,17}^- \pmod* 6.
	$$
}
\end{example} 
	We end this section by making some comments on Examples 5.6 and 5.7 considered 
	in \cite{dufr}. These involve congruences between the coefficients of a newform in $S_k(p)$ and $E_k$ away from the level $p$.
	More precisely,  in those examples, it is shown that if $\ell \mid \frac{B_k}{2k}(1-p^{k})$ and $\ell \ge 5,$ then there exists a newform $f\in S_k(p)$ such that, for all primes $q\neq p$, 
	modulo a prime ideal $\lambda$ above $\ell$ we have $a_f(q)\equiv 1+q^{k-1}\pmod*{\lambda}$. 
	On using \thmref{main} more can be said. For both 
	examples the prime $\ell$ divides  $\frac{B_k}{2k}(\varepsilon+p^{k/2})$ and $(\varepsilon+  p^{k/2}) (\varepsilon+ p^{k/2-1})$, for some $\varepsilon\in \{\pm 1\}$ and also satisfies 
	the further requirements of \thmref{main}, which then
	yields that the $W_p$-eigenvalue of $f$ is $\varepsilon$ and 
	that $a_f(p)\equiv 1+p^{k-1}+\varepsilon p^{k/2}\pmod*{\lambda}.$

	\section{Intermezzo: Anatomy of integers}\label{dickman}
	This section is a preamble for the next one.
	\par Let $P^+(n)$ denote the largest prime divisor
	of an integer $n\geqslant 2$. Put $P^+(1)=1$. 
	A number $n$ is said to be {\it $y$-friable}\footnote{Some authors use $y$-smooth. Friable
		is an adjective meaning easily crumbled or broken.} if $P^+(n)\leqslant y$. 
	The number of integers $1\leqslant n\leqslant x$
	such that $P^+(n)\leqslant y$ is denoted by $\Psi(x,y)$. 
	The study of the smoothness of integers was dubbed psixyology by the third
	author in his PhD thesis, but is currently called the anatomy of integers 
	(thus the focus shifted from the mind of numbers to their body...).\\
	\indent In 1930, Dickman \cite{Dickman} proved that
	\begin{equation}
		\label{dikkertje}
		\lim_{x\rightarrow \infty}\frac{\Psi(x,x^{1/u})}{x}=\rho(u),
	\end{equation}
	where the {\it Dickman function} 
	%(today often also called {\it Dickman-de Bruijn function})
	$\rho(u)$
	is defined by 
	\begin{equation}
		\label{defie}
		\rho(u)=
		\begin{cases}
			1 & \text{for $0\leqslant u\leqslant 1$};\\
			\frac{1}{u}\int_{0}^1 \rho(u-t) dt & \text{for $u>1$.}
		\end{cases}
	\end{equation}
	We have $0<\rho(u)<1/\Gamma(u+1),$ where $\Gamma$ is the Gamma function. Thus
	$\rho$ is rapidly decreasing. Dickman's result also remains true if we ask
	for the proportion of integers $n\le x$ such that $P^+(n)\le n^{1/u}.$ 
	Thus if $p$ is a prime number and $p-1$ would
	behave like a typical integer, then one would expect that  $P^+(p-1)\le p^{1/u}$ 
	with probability $\rho(u).$ The following 
	known result
	partially confirms this.
	\begin{thm}
		\label{pu}
		Let $s$ be any fixed non-zero integer.
		The set $\{p:P^+(p+s)\ge p^{1/u}\}$ 
		has density $1-\rho(u)$ under the Elliott-Halberstam
		conjecture and unconditionally a lower 
		density at least 
		$1-4\rho(u)$ for $u>u_1,$ where $u_1\in (2.677,2.678)$ is the unique solution of
		the equation $4u_1\rho(u_1)=1.$ 
		%Let $u<1$ be fixed.
		%Under the Elliott-Halberstam conjecture we have $\pi_u(x)\sim \rho(u)\sum_{p\le x}1$.
		%Unconditionally we have $\pi_u(x)\ge (1-4\rho(u)-\varepsilon))\sum_{p\le x}1$.
	\end{thm}
	\begin{proof}
		A detailed
		proof of the first
		assertion was given by Lamzouri \cite{lam} and, independently, by Wang \cite{wang}. The second assertion is
		due to Feng-Wu \cite{FW} and Liu-Wu-Xi \cite{liwuxi}.
	\end{proof}
	Unconditionally the set considered in 
	\thmref{pu}
	is not known to have a density, and therefore we
	work with the notion of lower and upper density.
	\par Interestingly, \eqref{dikkertje} was 
	already known
	to Ramanujan, for details and more about the
	behaviour of the Dickman function, see, e.g., 
	Moree \cite{mordeB}.
	
	\section{Application to the degree of the coefficient field}\label{application}
	Let $K_f$ be the coefficient field of 
	a normalized eigenform $f\in S_k(p).$ Put
	$$d_k^{\rm new}(p):= \max\{[K_f:\mathbb Q]:f\in S_k(p), ~f\text{~newform}\}$$
	and
	$$d_k(p):= \max\{[K_f:\mathbb Q]:f\in S_k(p), ~f\text{~normalized~eigenform}\}.$$
	Billerey-Menares \cite{bime} showed that for every even integer $k\ge 2$ and every prime $p\ge (k+1)^4$ 
	with $P^+(p-1)\ge 5$ one has 
	\begin{equation}
		\label{5/2}
		d_k^{\rm new}(p) \ge  \frac{5\log(P^+(p-1))}{2k}.
	\end{equation}
	(Actually the original result has the 
	factor $\log(1+2^{(k-1)/2})$ in the denominator, which we prefer to
	replace by the upper bound $k/5.$) In 2015, Luca et al.\,\cite{LMP} 
	showed that there exists a set of primes of natural density at least $3/4$ such 
	that $P^+(p-1)\ge p^{1/4}.$ This result in combination with 
	inequality \eqref{5/2}, then yields that 
	\begin{equation}
		\label{oud}
		d_k^{\rm new}(p) \ge \frac{5\log p }{8 k} 
	\end{equation}
	for a set of primes of density at least $3/4.$ 
	If one wants a lower bound valid for all large
	enough primes $p,$ we still cannot do better than
	Bettin et al.\,\cite{verynew} who showed that 
	$$
	d_k^{\rm new}(p) \gg_k \log \log p,~~ p\rightarrow \infty.
	$$
	On combining \thmref{pu} and inequality 
	\eqref{5/2}, we obtain the following
	improvement of \eqref{oud}.
	%Making use of recent developments on large prime factors of shifted primes, we can 
	%improve on this.
	\begin{thm}
		Let $k\ge 2$ be an even integer and $u>1$ any real number. Under the Elliott-Halberstam
		conjecture the set of primes $p$ for which
		$$
		d_k^{\rm new}(p) \ge \frac{5\log p }{ 2 u k} 
		$$
		has lower density at least
		$1-\rho(u).$ Unconditionally this
		set has 
		at least lower density $1-4\rho(u)$ for
		$u>2.678.$ 
	\end{thm}
	\begin{cor}
		The  set of 
		primes $p$ for which 
		\eqref{oud} holds has lower
		density at least $1-4\rho(4)\ge 0.98.$
	\end{cor}
	In the same spirit and as an application of 
	 \thmref{eigenform}, we establish analogues 
	for $d_k(p)$ of
	the latter theorem and corollary,   
	by following the ideas used in the proof of \cite[Theorem 2]{bime}.
	%\textcolor{red}{Shall we mention that if $p$ is different from 2, 3, 5, 7 and 17 then $P^+(p^2-1) \ge 5$.}
	\begin{thm}
		\label{dkpbound}
		If $k\ge 4$ is an even integer and $p$ any prime, then
		$$d_k(p) \ge \frac{5\log (P^+(p^2-1)) }{2k}.$$
	\end{thm}
	\begin{proof}
		Define $\ell:= P^+(p^2-1).$ First
		assume that $\ell \ge 5.$ Then we can choose $\varepsilon\in \{\pm 1\}$ such that $\ell \mid (\varepsilon+  p^{k/2})$. Applying \thmref{eigenform} yields a normalized eigenform $f=\sum_{n\ge 1}a_f(n)q^n \in S_k (p)$ and a prime ideal $\lambda \subset \mathcal{O}_{K_f}$ above $\ell$ such that 
		$$
		f\equiv E_{k,p}^\varepsilon \pmod* \lambda.
		$$
		The coefficients of the eigenform $f$ and those of any Galois conjugate of $f$ all 
		satisfy Deligne's estimate and hence the non-zero algebraic integer $b:=a_f(2)-1-2^{k-1}$ and all its Galois conjugates have absolute value
		bounded above by $(1+2^{(k-1)/2})^2$.  Because of the above congruence, it is clear that $b\in \lambda$ and hence $\ell$ divides %$|N_{K_f/\mathbb Q}(b)|$ (
		the absolute value of the norm of $b,$ which is at most $(1+2^{(k-1)/2})^{2d}$, where  $d=[K_f:\mathbb Q]$. Therefore we conclude that
		$$
		d_k(p)\ge d\ge \frac{\log \ell}{2\log \left(1+2^{(k-1)/2} \right)}\ge \frac{5\log \ell}{2k}.
		$$
		In the remaining case $\ell\le 3$ we have 
		$5(\log \ell)/2k\le 5(\log 3)/8<1\le d_k(p),$ and there is nothing to prove.
	\end{proof}
	Combining this with \thmref{pu} we obtain the following corollary. 
	\begin{cor}
		Let $k\ge 4$ be even integer  and $u>1$ any real number. The set of primes $p$ for which
		$$
		d_k(p) \ge \frac{5\log p }{ 2 u k} 
		$$
		has lower density at least $1-4\rho(u)$ for
		$u>2.678.$ 
	\end{cor}
	This can be sharpened if
	one solves the
	following problem.
	\begin{prob}
		Show that the set of primes
		$p$ for which $P^+(p^2-1)<p^{1/u}$ 
		has an upper density
		not exceeding $4\rho(u)$ for all $u$ large
		enough.
	\end{prob}	
	By \thmref{pu} under the Elliott-Halberstam
	conjecture each of the inequalities
	$P^+(p-1)< p^{1/u}$ and $P^+(p+1)<p^{1/u}$ 
	is satisfied with probability $\rho(u).$
	Assuming
	independence of the two events leads to the following
	conjecture on invoking Theorem \ref{dkpbound}.
	\begin{con}
		Let $k\ge 4$ be even integer and $u>1$ any real number. 
		The set of primes $p$ for which
		$$d_k(p) \ge \frac{5\log p }{ 2 u k},$$
		has lower density at least $1-\rho(u)^2.$ 
	\end{con}
	The independence assumption was already made
	by Erd\H{o}s and Pomerance, see the two articles
	by Wang \cite{wang,wang3}, who also made partial
	progress in proving it.
	\par Using elementary number theory it is easy to show that
	$P^+(p^2-1)\le 3$ if and only if $p\in \{2,3,5,7,17\}.$ A much deeper
	result is the following trivial
	consequence of a celebrated 
	result of Evertse (where as usual we denote by $\pi(x)$ the number
	of primes $p\le x$).
	\begin{prop}
		Let $x$ be a real number.
		The number of primes $p$ for which
		$P^+(p^2-1)\le x$ is at most $3\cdot 7^{1+2\pi(x)}.$
	\end{prop}
	\begin{proof} Put $S:=\{p_1,\ldots,p_s\}.$ Integers of the
		form $\pm p_1^{e_1}\cdot p_2^{e_1}\cdots p_s^{e_s}$ are
		called $S$-units.
		Evertse 
		\cite[Theorem 1]{Ever} proved that the equation 
		$a+b=1$ has at most $3\cdot 7^{1+2s}$ solutions 
		$(a,b)$ with $a$ and $b$ both S-units.
		If $P^+(p^2-1)\le x,$ then both
		$p-1$ and $p+1$ are $S$-units, with $S$ the 
		set of consecutive primes not excedding $x,$ 
		which has cardinality $\pi(x)$. Taking the difference
		and dividing by two leads to the S-unit equation $a+b=1.$
	\end{proof}
	Combining this result with \thmref{dkpbound}, we obtain the following result.
	\begin{thm}
		Let $m\ge 1$ be an integer. 
		Then, with the exception of at most 
		$O(e^{e^{2km/5}})$ primes $p,$ 
		we have
		$d_k(p)\ge m.$
	\end{thm}
	Proving a similar result for $d_k^{\rm new}(p)$ is an open problem.
	\section{Proof of \thmref{non-divisibility}}
	\label{pietersproof}
	A set of natural numbers $S$ is said to be \emph{multiplicative} if its
	characteristic function is a multiplicative function. 
	The set $B$ of natural numbers that can be written as a sum of two squares provides an example (as 
	already Fermat 
	knew). 
	One can 
	wonder about the asymptotic behavior of $S(x)$,  
	the number of positive integers $n\le x$ that are in $S$.
	An important role in understanding $S(x)$ is played by the Dirichlet series 
	\begin{equation}
		\label{LS}
		L_S(s) := \sum_{n\in S}n^{-s},
	\end{equation}
	which converges for $\Re(s) > 1$.
	If the limit
	\begin{equation}
		\label{EKf} 
		\gamma_S:=\lim_{s\rightarrow 1^+}\bigg(
		\frac{L'_S(s)}{L_S(s)} 
		+\frac{\alpha}{s-1}\bigg)
	\end{equation}
	exists for some $\alpha\ne 0$, we say that the set $S$ admits an \emph{Euler-Kronecker constant} $\gamma_S$. A leisurely account
	of the theory of Euler-Kronecker constants with plenty of examples
	is given in Moree \cite{morSx}.
	
	\begin{proof}[Proof of Theorem \ref{non-divisibility}]
		We first consider the non-divisibility of $\sigma_{k-1}(n)$ by
		$\ell.$ Note that $\ell\nmid \sigma_{k-1}(n)$ if and only if
		$\ell\nmid \sigma_r(n),$ where $r=(k-1,\ell-1).$
		Let $g_{p_1}$ be the multiplicative order
		of $p_1^{r}$ modulo $\ell,$ with $p_1\ne \ell$ an arbitrary
		prime number.
		We let
		\begin{equation}
			\label{gp-def}
			\mu_{p_1}=\begin{cases}  \ell&\text{if~}g_{p_1}=1,\\
				g_{p_1} &\text{if~}g_{p_1}>1.\end{cases}
		\end{equation}
		Put $S_1:=\{n:\ell\nmid \sigma_r(n)\}.$
		Rankin \cite{Ts} showed that
		\begin{equation}
			\label{EulerProductTsgeneral1}
			L_{S_1}(s)=\frac1{1-\ell^{-s}}\prod_{p_1\ne \ell}\frac{1-p_1^{-(\mu_{p_1}-1)s}}{(1-p_1^{-s})(1-p_1^{-\mu_{p_1}s})}.
		\end{equation}
		The proof of this is short and 
		can also be found in \cite{clm}.
		Set $f(n)=\sigma_{k-1}(n)+\varepsilon p^{k/2}\sigma_{k-1}(n/p).$ Write
		$n=mp^e,$ with $m$ coprime to $p.$ Note that $f(m)=\sigma_{k-1}(m).$
		Writing momentarily $\sigma$ instead of $\sigma_{k-1},$ we have
		$$f(mp^e)=\sigma(mp^e)+\varepsilon p^{k/2}\sigma(mp^{e-1})=\sigma(m)\sigma(p^e)+\varepsilon p^{k/2}\sigma(m)\sigma(p^{e-1})=f(m)f(p^e),$$
		and so $f$ is a multiplicative function. Put $S_2(p):=\{n:\ell\nmid f(n)\}$.
		The multiplicativity of $f$ implies that
		$$L_{S_2(p)}(s)=\prod_{p_1}\sum_{\substack{a\ge 0\\ \ell\nmid f(p_1^a)}}\frac{1}{p_1^{as}}$$
		has an Euler product, the sums being the Euler product
		factors.
		For the primes $p_1\ne p,$ we have
		$\ell\nmid f(p_1^a)$ if and only if $\ell\nmid \sigma_r(p_1^a)$ and we get the same Euler product factors as 
		in \eqref{EulerProductTsgeneral1}.
		Since $\varepsilon p^{k/2}=-1\pmod*{\ell}$ by assumption, we have
		$$f(p^a)\equiv 
		\sigma_{k-1}(p^a)-\sigma_{k-1}(p^{a-1})\equiv p^{a(k-1)}\not\equiv 0\pmod*{\ell},$$ and we 
		conclude that the Euler product factor at $p_1=p$ is $(1-p^{-s})^{-1}$.
		We infer that
		\begin{equation}
			\label{EulerProductTsgeneral2}
			L_{S_2(p)}(s)=\frac1{(1-\ell^{-s})}\frac1{(1-p^{-s})}\prod_{\substack{p_1\ne \ell\\ p_1\ne p}}\frac{1-p_1^{-(\mu_{p_1}-1)s}}{(1-p_1^{-s})(1-p_1^{-\mu_{p_1}s})}=
			L_{S_1}(s)\frac{(1-p^{-\mu_{p}s})}{1-p^{-(\mu_{p}-1)s}}.
		\end{equation}
		Comparing the logarithmic derivative of $L_{S_2(p)}(s)$ with that of 
		$L_{S_1}(s),$ we
		obtain
		$$
		\frac{L'_{S_2(p)}(s)}{L_{S_2(p)}(s)}=
		\frac{L'_{S_1}(s)}{L_{S_1}(s)}+\log p\left(\frac{\mu_p}{p^{\mu_ps}-1}-\frac{(\mu_p-1)}{p^{(\mu_p-1)s}-1}\right).
		$$
		The proof of \eqref{gammarelator} is then completed on invoking \eqref{EKf}. 
		\par The prime counting functions
		$\#\{p_1\le x:~\ell\mid \sigma_r(p_1)\}$ and
		$\#\{p_1\le x:~\ell\mid f(p_1)\}$
		differ
		by at most one for every $x.$ As the first counting function
		satisfies \eqref{primecondition} with
		$\delta=r/(\ell-1),$ so does the second, and the
		proof is completed on account of Theorem \ref{conditionA}.
	\end{proof}

	\noindent \textbf{Acknowledgment.} The authors would like to thank Shaunak Deo for several useful discussions, going through the paper carefully and giving useful suggestions. We also thank Jaban Meher for sharing his ideas in the proof of \thmref{main2} and Dan Fretwell for his useful comments. Matteo Bordignon 
	explained the authors how to compute $\rho(u)$ using SAGE, Zhiwei Wang updated them on the literature on friable shifted primes and Nicolas Billerey and Ricardo Menares pointed
	out the existence of \cite{verynew}. The third 
	author thanks Younes Nikdelan for his meticulous proofreading of earlier
	versions and frequent discussions on his related preprint \cite{you}. The authors thank the anonymous referees for a careful reading of the manuscript and  for giving valuable suggestions.
	
	\par The research of the first author was supported by the grant no. 692854 provided by the European Research Council (ERC) while the second author was supported by Israeli Science Foundation grant  1400/19. This work was carried out when these two authors were postdoctoral fellows at Hebrew University of Jerusalem and Bar-Ilan University respectively. 
	\par The open-source mathematics software SAGE (www.sagemath.org) has been used for numerical computations in this work.
%\vfil\eject	
	
\end{document}